\documentclass[11pt]{article}
\usepackage{amsmath,amsthm,amsfonts}
\usepackage[numbers]{natbib}
\usepackage[colorlinks,ps2pdf]{hyperref}
\usepackage[numbers]{natbib}
\usepackage{graphicx,psfrag}
\usepackage{enumerate}

\theoremstyle{plain}
\newtheorem{theorem}{Theorem}
\newtheorem{lemma}{Lemma}
\newtheorem{proposition}{Proposition}
\theoremstyle{definition}
\newtheorem{definition}{Definition}
\newtheorem{example}{Example}

\newcommand{\R}{\mathbb{R}}
\newcommand{\Z}{\mathbb{Z}}
\newcommand{\geom}{\operatorname{geom}}
\newcommand{\reff}[1]{(\ref{#1})}
\newcommand{\fref}[1]{(F\ref{#1})}
\newcommand{\sE}{\operatorname{E}}
\newcommand{\sP}{\operatorname{P}}

\begin{document}
\title{Stabilization of an overloaded queueing network\\
       using measurement-based admission control\thanks{First published in Journal of Applied Probability
       43(1):231--244. \copyright\ 2006 by the Applied Probability Trust.}}

\author{Lasse Leskel\"a\thanks{
 Institute of Mathematics, Helsinki University of Technology,
 P.O. Box 1100, FI-02015 TKK, Finland. \url{http://www.iki.fi/lsl/}}
}

\date{}

\maketitle

\begin{abstract}
 Admission control can be employed to avoid congestion in queueing
 networks subject to overload. In distributed networks the admission
 decisions are often based on imperfect measurements on the network
 state.  This paper studies how the lack of complete state
 information affects the system performance by considering a simple
 network model for distributed admission control.  The stability
 region of the network is characterized and it is shown how feedback
 signaling makes the system very sensitive to its parameters.
\end{abstract}

\


\


\section{Introduction}
Consider an overloaded queueing network where the incoming traffic
exceeds the service capacity over a long time period. In this case it
is often necessary to employ admission control to avoid the network to
become fully congested. Many networks of practical interest are
composed of subnetworks, not all of which are administered by a single
party.  In such a network the admission controller seldom has complete
up-to-date system information available. Instead, the admission
decisions must be based on partial measurements on the network state.

This paper studies the effect of imperfect information to the
performance of the admission control scheme.  Typical performance
measures for well-dimensioned networks in this kind of setting include
the average amount of rejected traffic per unit time, and the mean
proportion of time the network load is undesirably high. However,
assuming the network under study is subjected to long-term overload,
there is another performance criterion that must first be analyzed,
namely: \emph{If the network is subjected to a stationary load
  exceeding the service capacity, how strict admission control rules
  should one set in order to stabilize the system?}

To deal with the question mathematically, it is assumed that the
network can be modeled using the simplest nontrivial model for a
distributed network, the two-node tandem network with independent and
exponential service times and unlimited buffers.  The network state is
denoted by $X=(X_1,X_2)$ where $X_i$ is the number of jobs in
node~$i$.  It is assumed that the admission control can be modeled so
that the input to the system is a Poisson process with a stochastic
time-varying intensity, the intensity $\lambda=\lambda(X)$ being a
function of the network state.

The lack of complete state information is reflected in the model by
assuming that the input rate $\lambda$ is a function of only one of
the $X_i$.  If $\lambda(X) = \lambda(X_1)$, then the analysis of the
system can be reduced to the study of birth--death processes, which
are well understood.  This is why in the following it is always
assumed that $\lambda(X) = \lambda(X_2)$, so that the admission
control introduces a feedback signaling loop to the system. For
example, one can model a network where arriving traffic is rejected
when the size of the second buffer exceeds a threshold level $K$ by
setting $\lambda(X) = 1(X_2\leq K)$, see Figure~\ref{fig:systemGraph}.
In order to also cover more complex admission policies with multiple
thresholds and thinning of input traffic, the shape of $\lambda(X_2)$
will not be restricted in any way.

\begin{figure}[h]
  \begin{center}
    \psfrag{1(X2<=K)}{$1(X_2\leq K$)}
    \psfrag{X1}{$X_1$}
    \psfrag{X2}{$X_2$}
    \psfrag{K}{$K$}
    \psfrag{AC}{$AC$}
    \includegraphics[width=.9\textwidth]{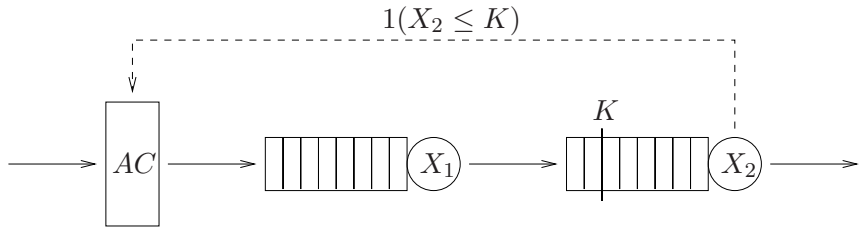}
    \caption{Admission control based on the single threshold level $K$.}
    \label{fig:systemGraph}
  \end{center}
\end{figure}

More precisely, $X$ is defined as a continuous-time stochastic process
as follows. Let $\lambda$ be a nonnegative function on $\Z_+$ and
$\mu_1,\mu_2>0$. Define the transition rates $q(x,y)$ for $x\neq y$,
$x,y\in\Z_+^2$ by
\begin{equation}
  \label{eq:qDef}
  q(x,y) = \left\{
    \begin{aligned}
      \lambda(x_2), \quad &y = x+e_1,\\
      \mu_1       , \quad &y = x-e_1+e_2\geq 0,\\
      \mu_2       , \quad &y = x-e_2\geq 0,\\
      0           , \quad &\text{otherwise},
    \end{aligned}
  \right.
\end{equation}
where $e_i$ denotes the $i$-th unit vector of $\Z_+^2$. As usual, set
$q(x,x)=-q(x)$, where the transition rate out of state $x$ is defined by
\[
q(x) = \sum_{y\neq x} q(x,y).
\]
It is clear that $q(x)<\infty$ for all $x$, so using the minimal
construction~\cite{asmussen2003,bremaud1999} the rates $q(x,y)$ define
a unique Markov process $X$ on $\Z_+^2\cup\{\kappa\}$.  Here $\kappa$
denotes an additional state not in $\Z_+^2$ with $T_\kappa =
\inf\{t>0: X(t)=\kappa\} \leq \infty$ being the time of explosion of
$X$.  The notation $S(\lambda,\mu_1,\mu_2)$ will be used for the set
of transition rates corresponding to the the triple
$(\lambda,\mu_1,\mu_2)$, and the system $S(\lambda,\mu_1,\mu_2)$ is
said to be \emph{stable} if the corresponding Markov process is
ergodic, that is, irreducible and positive recurrent.

In its most general form, the stability problem may now be stated as
\begin{description}
\item [(P1)] Characterize the set of all $(\lambda,\mu_1,\mu_2)\in
  \R_+^{\Z_+}\times\R_+\times\R_+$ for which $S(\lambda,\mu_1,\mu_2)$
  is stable.
\end{description}
Specializing to networks with threshold-based admission control, the
offered traffic is assumed to arrive at unit rate, without loss of
generality.  Denoting the admission threshold by $K$,~(P1) now takes
the form
\begin{description}
\item[(P2)] For each $(\mu_1,\mu_2) \in \R_+^2$,
  determine for which values of $K\in\Z_+\cup \{\infty\}$, if any, the
  system $S(\, 1(\cdot\leq K), \, \mu_1,\mu_2)$ is stable.
\end{description}
Note that the system corresponding to $K=\infty$ in~(P2) is the
ordinary tandem queue, for which it is well-known that
$\min(\mu_1,\mu_2)>1$ is sufficient and necessary for stability. On
the other hand, assuming overload, answering the question on the
existence of a threshold level that can stabilize the system is not as
straightforward.

The queueing systems literature includes a vast amount of work on
various admission control mechanisms.  However, most earlier studies
on tandem networks require at least one of the buffers to be finite,
so that the two-dimensional nature of the problem can partly be
reduced to one-dimensional by applying matrix-geometric
methods~\cite{neuts1981}.  For networks with unlimited buffers and
state-dependent service times, Bambos and Walrand~\cite{bambos1989}
provide stability results extending to non-Markovian systems, however
ruling out networks with the type of feedback signaling loop present
here.  Concerning the network $S(\lambda,\mu_1,\mu_2)$ defined above,
the compensation approach introduced by Adan, Wessels, and
Zijm~\cite{adan1993} can be used for computing the invariant measure
in the special case where $\lambda$ is constant on
$\{n\in\Z_+:n\geq1\}$. For more general input rates, Leskel\"a and
Resing~\cite{leskela2007} have described a numerical method for
calculating stationary performance characteristics of the system.
Altman, Avrachenkov, and N\'u\~nez Queija~\cite{altman2004} have
recently introduced perturbation techniques that seem appropriate for
asymptotically analyzing the behavior of $S(\lambda,\mu_1,\mu_2)$
under suitable parameter scaling.

This paper partially answers~(P1) by deriving sufficient and necessary
conditions for stability. Furthermore, by showing that in the special
case of threshold-based admission control the sufficient and necessary
conditions coincide, a complete solution of~(P2) is given.  In
addition, the sensitivity of the system is analyzed with respect to
changes in the service rates and it is shown how acceleration of one
of the servers may, rather paradoxically, destabilize the system.

\section{A sufficient condition for stability}
Let $S$ be a countable set.  For a function $V: S\to\R$,
denote
\[
\lim_{x\to\infty} V(x) = \infty
\]
if the set $\{x:V(x)\leq M\}$ is finite for all $M\in\R$. Further, the
mean drift of $V$ with respect to transition rates $q(x,y)$ is denoted
by
\begin{equation}
  \label{eq:driftDef}
  \Delta V(x) = \sum_{y\neq x} (V(y)-V(x)) \, q(x,y),
\end{equation}
assuming the sum on the right-hand side converges.

\begin{definition}
  A map $V:S\to\R$ is called a \emph{Lyapunov function} for $q$ if it
  satisfies the following conditions called \emph{Foster's criteria}:
  \begin{enumerate}[(F1)]
  \item \label{FosterFin} $\sum_{y\neq x} |V(y)-V(x)| \, q(x,y) < \infty$ for all $x$ (so that
    the right-hand side of~\reff{eq:driftDef} makes sense).
  \item \label{FosterLim} $\lim_{x\to\infty} V(x) = \infty$.
  \item \label{FosterDrift} There is a finite set $S_0\subset S$ such
    that $\, \sup_{x\in S\setminus S_0} \Delta V(x)<0$.
  \end{enumerate}
\end{definition}

The following continuous-time analogue of Foster's classical
theorem~\cite{foster1953} provides a sufficient condition for
stability.
\begin{theorem}[Tweedie~\cite{tweedie1975}]
  \label{the:foster}
  Let $X$ be an irreducible Markov process on a countable state space
  $S$ generated by transition rates $q(x,y)$ so that $q(x)<\infty$ for
  all $x$. The existence of a Lyapunov function for $q$ is then
  sufficient for $X$ to be ergodic.
\end{theorem}

Considering the system $S(\lambda,\mu_1,\mu_2)$, let $q(x,y)$ be as
defined in~\reff{eq:qDef}.  Assume $V$ is a function on $\Z_+^2$ of
the form $V(x) = x_1 + v(x_2)$ for some $v:\Z_+\to\R$ with $v(0)=0$.
Searching for a Lyapunov function of this type, let us fix a number
$r>0$ and require that the mean drift of $V$ with respect to $q$
satisfies
\begin{equation}
  \label{eq:VRequirement}
  \Delta V(x) = -r \quad \text{for all $x$ with $x_1>0$}.
\end{equation}
It is straightforward to verify that~\reff{eq:VRequirement} is
equivalent to
\begin{align*}
  v(1)   &= 1-(\lambda(0)+r)/\mu_1,\\
  v(n+1) &= 1-(\lambda(n)+r)/\mu_1 + (1+\mu_2/\mu_1)v(n) - \mu_2/\mu_1 v(n-1), \quad n\geq 1.
\end{align*}
Denoting $\alpha(n)=1-(\lambda(n)+r)/\mu_1$ and $w(n)=v(n+1)-v(n)$,
the above difference equation can be written as $w(n) = \alpha(n) +
\mu_2/\mu_1 \, w(n-1)$ for $n\geq 1$, with $w(0)=\alpha(0)$. Thus,
$w(n) = \sum_{k=0}^n \alpha(k) \, (\mu_1/\mu_2)^{k-n}$, so that
\[
v(n) = \sum_{j=0}^{n-1} w(j)
= \sum_{j=0}^{n-1}\sum_{k=0}^j \alpha(k) \, (\mu_1/\mu_2)^{k-j},
\]
and we conclude that~\reff{eq:VRequirement} defines for each $r>0$ the
function
\[
V_r(x) = x_1 + \sum_{j=0}^{x_2-1} \sum_{k=0}^{j} (1-(\lambda(k)+r)/\mu_1)) \, (\mu_1/\mu_2)^{k-j}.
\]
Thus we have constructed a family of functions $\mathcal{V} = \{V_r:
r>0\}$ whose elements satisfy $\sup_{x:x_1>0} \Delta V_r(x)<0$, so
there are hopes that $V_r$ might satisfy~\fref{FosterDrift} for a
suitably chosen finite subset of $\Z_+^2$. In order to investigate
whether this is the case, let us study the mean drift of $V_r$ for
$x=(0,n)$ with $n\geq 1$,
\begin{equation}
  \label{eq:boundaryDrift}
  \Delta V_r(0,n) = \lambda(n) - \mu_2(v_r(n)-v_r(n-1)).
\end{equation}

\begin{definition}
  For $z\geq 0$, denote $Z_n\sim\geom_n(z)$ if $Z_n$ is a random
  variable on $\Z\cap [0,n]$ with $\sP(Z_n=j) = c z^j$. For $0\leq z
  \leq 1$, denote $Z\sim\geom(z)$ if the random variable $Z$ on
  $\Z_+$ satisfies $P(Z=j) = (1-z) z^j$.
\end{definition}
In this paper, $Z_n$ and $Z$ will always represent generic random
variables with distributions $\geom_n(\mu_1/\mu_2)$ and
$\geom(\mu_1/\mu_2)$, respectively. Using this notation, one may
verify that~\reff{eq:boundaryDrift} can be alternatively written as
\begin{equation}
  \label{eq:boundaryDrift2}
  \Delta V_r(0,n) = \frac{\sE\lambda(Z_n)-\mu_2(1-r/\mu_1) \,
  \sP(Z_n>0)}{\sP(Z_n=n)}, \quad Z_n\sim\geom_n(\mu_1/\mu_2).
\end{equation}

\begin{theorem}
  \label{the:suffCond}
  The family $\mathcal{V}=\{V_r:r>0\}$ contains a Lyapunov function for
  $S(\lambda,\mu_1,\mu_2)$ if and only if
  \begin{equation}
    \label{eq:suffCond}
    \varlimsup \sE\lambda(Z_n) < \min(\mu_1,\mu_2), \quad Z_n\sim\geom_n(\mu_1/\mu_2).
  \end{equation}
  In particular, if $\lambda(0)>0$, then~\reff{eq:suffCond} is
  sufficient for the stability of $S(\lambda,\mu_1,\mu_2)$.
\end{theorem}

The proof of the theorem will utilize the following two lemmas.

\begin{lemma}
  \label{lem:suffCond}
  Condition~\reff{eq:suffCond} is equivalent to
  $\varlimsup \Delta V_r(0,n) < 0$ for some $r>0$.
\end{lemma}
\begin{proof}
  Let $Z_n\sim\geom_n(\mu_1/\mu_2)$ for $n\geq 0$.  Observe first that
  since $\lim\sP(Z_n>0) = \min(1,\mu_1/\mu_2)$,
  \begin{equation}
    \label{eq:minToProb}
    \varlimsup\sE\lambda(Z_n) - \min(\mu_1,\mu_2)
    = \varlimsup \, \{ \sE\lambda(Z_n)-\mu_2\sP(Z_n>0) \}.
  \end{equation}
  Assume now that~\reff{eq:suffCond} holds. Then we can choose an
  $r>0$ so that $\sE\lambda(Z_n)-\mu_2\sP(Z_n>0) \leq -r$
  for $n$ large enough. It follows that $\varlimsup\Delta V_r(0,n)<0$, since
  using~\reff{eq:boundaryDrift2} we see that eventually for large $n$,
  \[
  \Delta V_r(0,n) \leq \frac{-r + r\, \mu_2/\mu_1 \sP(Z_n>0)}{\sP(Z_n=n)} = -r.
  \]

  For the other direction, assume $\varlimsup\Delta V_r(0,n)<0$ for
  some $r>0$. Then there is an $s\in(0,r)$ so that for $n$ large
  enough, $\Delta V_r(0,n)\leq -s$, and
  applying~\reff{eq:boundaryDrift2},
  \[
  \frac{\sE\lambda(Z_n)-\mu_2(1-s/\mu_1) \, \sP(Z_n>0)}{\sP(Z_n=n)}
  \leq \Delta V_r(0,n)
  \leq -s.
  \]
  This shows that
  \[
  \sE\lambda(Z_n)-\mu_2\sP(Z_n>0) \leq
  -s(\sP(Z_n=n)+\mu_2/\mu_1\sP(Z_n>0)) = -s
  \]
  for all $n$ large enough, and in light of~\reff{eq:minToProb} it
  follows that $\varlimsup\sE\lambda(Z_n) < \min(\mu_1,\mu_2)$.
\end{proof}

\begin{lemma}
  \label{lem:infinity2D}
  Let $f$ be a function of the form $f(x)=u(x_1)+v(x_2)$ for some
  $u,v:\Z_+\to\R$.  Then $\, \lim_{x\to\infty} f(x) = \infty $ if and
  only if $\, \lim_{x_1\to\infty} u(x_1) = \infty$ and $\,
  \lim_{x_2\to\infty} v(x_2) = \infty$.
\end{lemma}
\begin{proof}
  Assume $\lim u(x_1) = \lim v(x_2) = \infty$, and fix an $M\in\R$.
  Since $u_0 = \inf u(x_1)$ and $v_0 = \inf v(x_2)$ are finite, we can
  choose $m_1$ and $m_2$ such that $u(x_1) > M - v_0$ for all $x_1>
  m_1$, and $v(x_2) > M - u_0$ for all $x_2> m_2$. Hence, $f(x)>M$ if
  either $x_1>m_1$ or $x_2>m_2$, so that the set $\{x: f(x)\leq M\}
  \subset [0,m_1]\times [0,m_2]$ is finite. Since $M$ was arbitrary,
  $\lim_{x\to\infty} f(x) = \infty$.

  Suppose next that $\lim_{x\to\infty} f(x) = \infty$. Then if
  $\varliminf u(x_1) < \infty$, there is a $c\in \R$ so that $S =
  \{x_1: u(x_1) \leq c\}$ is infinite.  This implies that $\{x: f(x)
  \leq c+v(0)\} \supset S\times \{0\}$ is infinite, contrary to the
  assumption $\lim_{x\to\infty} f(x) = \infty$.  Thus, $\varliminf
  u(x_1)=\infty$. Similarly, one proves that $\varliminf
  v(x_2)=\infty$.
\end{proof}

\begin{proof}[Proof of Theorem~\ref{the:suffCond}]
  Let $r>0$ and assume $V_r\in\mathcal{V}$ is a Lyapunov function for
  $q$.  Let $S_0$ be a finite set so that~\fref{FosterDrift} holds.
  Then $\{0\}\times (n_0,\infty) \subset S_0^c$ for some $n_0$, which
  implies
  \[
  \varlimsup \Delta V_r(0,n)
  \leq  \sup_{n>n_0} \Delta V_r(0,n)
  \leq \sup_{x\in S_0^c} \Delta V_r(x)
  < 0.
  \]
  By Lemma~\ref{lem:suffCond}, this implies~\reff{eq:suffCond}.

  For the other direction, assume that~\reff{eq:suffCond} holds.
  Applying Lemma~\ref{lem:suffCond}, we can pick an $r>0$ so that
  $\varlimsup\Delta V_r(0,n)<0$. Hence, there is an $n_0$ and an
  $\epsilon>0$ so that $\Delta V_r(0,n) \leq -\epsilon$ for all
  $n>n_0$.  Denoting $S_0 = \{0\}\times [0,n_0]$, it follows that
  \[
  \sup_{x\in S_0^c} \Delta V_r(x)
  = \max\{ \sup_{n>n_0} \Delta V_r(0,n), \sup_{x:x_1>0} \Delta V_r(x) \}
  \leq \max\{ -\epsilon, -r\} <0,
  \]
  since by the construction of $V_r$, $\Delta V_r(x)=-r$ for all $x$
  with $x_1>0$. Thus, $V_r$ satisfies~\fref{FosterDrift}.  Next,
  observe that using~\reff{eq:boundaryDrift},
  \[
  \lambda(n) - \mu_2 (v_r(n)-v_r(n-1)) = \Delta V_r(0,n) \leq -\epsilon
  \quad\text{for} \ n>n_0.
  \]
  This shows that $v_r(n)-v_r(n-1) \geq\epsilon/\mu_2$ eventually for
  large $n$, so that $\lim_{n\to\infty}v_r(n) = \infty$. By
  Lemma~\ref{lem:infinity2D}, we conclude that $V_r$
  satisfies~\fref{FosterLim}. Further,~\fref{FosterFin} holds
  trivially since the set $\{x: q(x,y)>0\}$ is finite for all $x$.
  Thus, $V_r$ is a Lyapunov function for $q$.  Finally, note that $X$
  is irreducible when $\lambda(0)>0$. Hence, application of
  Theorem~\ref{the:foster} now completes the proof.
\end{proof}

\section{Necessary conditions for stability}

Assume $\lambda(0)>0$ so that the system $S(\lambda,\mu_1,\mu_2)$ is
irreducible.  In the previous section we saw that
\[
\varlimsup \sE \lambda(Z_n) < \min(\mu_1,\mu_2), \quad Z_n\sim\geom_n(\mu_1/\mu_2),
\]
is sufficient for the stability of $S(\lambda,\mu_1,\mu_2)$.  This
section is devoted to studying whether the above condition is also
necessary for stability.

\subsection{Small perturbations of Markov processes}
This section studies how ergodicity is preserved under small
perturbations of generators of Markov processes. If $q(x,y)$ and
$q'(x,y)$ are generators of Markov processes on a countable state
space $S$, denote
\[
D(q,q') = \{ x: q(x,y)\neq q'(x,y) \ \text{for some} \ y \}
\]
and
\[
\overline{D}(q,q')
= D(q,q') \cup \{ y: q(x,y)>0 \ \text{or} \ q'(x,y)>0 \ \text{for some} \ x\in D(q,q')\}.
\]
Further, for $F\subset S$ let
\[
T_F = \inf\{t>0:X(t-)\neq X(t), \ X(t)\in F\},
\]
with the convention $\inf\emptyset = \infty$, and $T_x = T_{\{x\}}$
for $x\in S$.

\begin{lemma}
  \label{lem:ergodicEquiv}
  Let $X$ and $X'$ be irreducible Markov processes on a countable
  state space $S$ generated by $q(x,y)$ and $q'(x,y)$, respectively,
  with $q(x)$, $q'(x)<\infty$ for all $x$.  Assume that
  $\overline{D}(q,q')$ is finite.  Then $X$ is ergodic if and only if
  $X'$ is ergodic.
\end{lemma}
\begin{proof}
  By symmetry, it is sufficient to show that the ergodicity of $X'$
  implies that of $X$. So, assume $X'$ is ergodic, and let $x$ be a
  state in $D=D(q,q')$. Denote the first jump time of $X$ by $\tau =
  \inf\{t>0: X(t-)\neq X(t)\}$. By irreducibility, $\sE_x\tau <
  \infty$, so by the strong Markov property,
  \[
  \sE_x T_D  = \sE_x\tau + \sE_x ( \sE_{X(\tau)} T_D ; X(\tau)\notin D).
  \]
  Since $q(x,y)$ and $q'(x,y)$ coincide outside $D$, and
  $\sP_x(X(\tau)\in\overline{D})=1$,
  \[
  \sE_x ( \sE_{X(\tau)} T_D; \, X(\tau)\notin D)
  = \sE_x ( \sE_{X(\tau)} T'_D; \ X(\tau)\in\overline{D}\setminus D)
  \leq \sup_{y\in \overline{D}} \sE_y T'_D.
  \]
  Since $X'$ is ergodic and $\overline{D}$ is finite, so is the
  right-hand side in the above inequality, and we conclude $\sE_x T_D
  < \infty$. Because $X$ is irreducible, this property implies that
  $X$ is positive recurrent (Meyn~\cite{meyn1993}, Theorem 4.3:(ii)
  and Theorem 4.4).
\end{proof}

\subsection{Bottleneck at node 1}
Assume $\mu_1<\mu_2$. Intuition suggests that in this case the
stability of the system depends on whether or not the buffer content
at node~1 grows to infinity. Observe that during the periods of time
where node~1 remains busy, the input to node~2 is a Poisson process
with rate $\mu_1$. The approach here is to compare the original
process to a saturated system where node~2 gets input at rate $\mu_1$
also during the time periods where node~1 is empty, and show that the
stability regions for the two systems are close to each other. With
this goal in mind, let us introduce another model family denoted by
$S^N(\lambda,\mu_1,\mu_2)$.  Fix a nonnegative integer $N$, and define
for $x\neq y$,
\[
q^N(x,y) = q(x,y) \, + \,  \mu_1 1(x_1=0, \, x_2<N, \, y=x+e_2).
\]
It is clear that when $\lambda(0)>0$, the transition rates $q^N(x,y)$
define using the minimal construction an irreducible Markov process
$X^N$ on $\Z_+^2\cup\{\kappa\}$. By Lemma~\ref{lem:ergodicEquiv} we
know that the stability of $S^N(\lambda,\mu_1,\mu_2)$ is equivalent to
that of $S(\lambda,\mu_1,\mu_2)$. Further, by letting $N$ approach
infinity, $S^N(\lambda,\mu_1,\mu_2)$ will resemble a network where
node~2 receives stationary input at rate $\mu_1$.

\begin{lemma}
  \label{lem:SStar}
  Assume that $S^N(\lambda,\mu_1,\mu_2)$ is stable.  Then the
  stationary distribution of $X^N$ satisfies
  \begin{equation}
    \label{eq:SStarStability}
    \sE\lambda(X^N_2) = \mu_2 \sP(X^N_2>0) - \mu_1 \sP(X^N_1=0,X^N_2<N),
  \end{equation}
  \begin{multline}
    \label{eq:SStarBound}
    \sP(X^N_2=n) =  (\mu_1/\mu_2)^n \sP(X^N_2=0) \\
    - 1(n>N) \sum_{j=N}^{n-1} (\mu_1/\mu_2)^{n-j} \sP(X_1^N=0,X_2^N=j),
  \end{multline}
  and for all real-valued $f$ on $\Z_+$,
  \begin{equation}
    \label{eq:SStarCond}
    \sE (f(X^N_2) ; X^N_2 \leq N) = \sE f(Z_N) \, \sP(X^N_2\leq N),  \quad Z_N\sim\geom_N(\mu_1/\mu_2).
  \end{equation}
\end{lemma}
\begin{proof}
  Starting from the balance equations for $X^N$, it is not hard to
  check that $\sE \lambda(X_2) = \mu_1\sP(X^N_1>0)$, and
  \[
  \mu_1\sP(X^N_1>0) + \mu_1 \sP(X^N_1=0,X^N_2<N) = \mu_2\sP(X^N_2>0),
  \]
  showing that~\reff{eq:SStarStability} is true. Further, it is
  straightforward to verify that for all $n$,
  \[
  \sP(X^N_2=n+1) = \mu_1/\mu_2 \, \left[\sP(X^N_2=n) - 1(n\geq N)\sP(X^N_1=0,X^N_2=n)\right],
  \]
  from which~\reff{eq:SStarBound} and~\reff{eq:SStarCond} follow.
\end{proof}

\begin{theorem}
  \label{the:necCondSmallMu1}
  Assume $\mu_1<\mu_2$, and let $Z\sim \geom(\mu_1/\mu_2)$. Then
  \begin{align*}
    \sE\lambda(Z) &< \mu_1
    \implies S(\lambda,\mu_1,\mu_2) \ \text{is stable},\\
    \sE\lambda(Z) &> \mu_1
    \implies S(\lambda,\mu_1,\mu_2) \ \text{is unstable}.
  \end{align*}
\end{theorem}
\begin{proof}
  Let $Z_N\sim \geom_N(\mu_1/\mu_2)$. Because $\mu_1<\mu_2$, it
  follows that $\sE\lambda(Z_N) \to \sE\lambda(Z)$ as $N \to \infty$. The first
  statement now follows from Theorem~\ref{the:suffCond}.  To prove the
  second claim, assume that $S(\lambda,\mu_1,\mu_2)$ is stable. Then
  by Lemma~\ref{lem:ergodicEquiv}, so is $S^N(\lambda,\mu_1,\mu_2)$
  for each $N$.  Applying~\reff{eq:SStarStability}
  and~\reff{eq:SStarCond} we see that
  \begin{equation}
    \label{eq:ineqSmallMu1}
    \sE \lambda(Z_N) \sP(X^N_2\leq N)
    = \sE (\lambda(X^N_2); X^N_2\leq N)
    \leq \mu_2 \sP(X^N_2>0).
  \end{equation}
  Next, \reff{eq:SStarBound} implies
  \[
  \sP(X^N_2>N) \leq \sum_{n>N} (\mu_1/\mu_2)^n \quad \text{for all} \ N,
  \]
  so that $\lim\sP(X^N_2\leq N) = 1$. This observation combined
  with~\reff{eq:SStarBound} implies
  \[
  \sP(X^N_2=0) = \frac{\sP(X^N_2\leq N)}{\sum_{n=0}^N (\mu_1/\mu_2)^n}
  \longrightarrow 1-\mu_1/\mu_2,
  \]
  as $N\to\infty$. Hence, $\lim\sP(X^N_2>0)=\mu_1/\mu_2$.  Letting $N
  \to \infty$ on both sides of~\reff{eq:ineqSmallMu1} now shows that
  $\sE \lambda(Z) \leq \mu_1$.
\end{proof}

\subsection{Bottleneck at node 2}
To study necessary stability conditions for the system when $\mu_1\geq
\mu_2$, the following asymptotical property of truncated geometric
random variables will be useful.
\begin{lemma}
  \label{lem:geomDistDiverging}
  Let $Z_n\sim\geom_n(z)$ with $z\geq 1$. Then for all nonnegative
  functions $f$ on $\Z_+$,
  \[
  \varliminf f(n)\leq \varliminf \sE f(Z_n) \leq \varlimsup\sE f(Z_n) \leq \varlimsup f(n).
  \]
\end{lemma}
\begin{proof}
  Without loss of generality, assume $\varlimsup f(n)<\infty$.  Choose
  a number $r$ so that $\varlimsup f(n) < r$. Then there is an $n_0$
  so that $f(n)\leq r$ for all $n>n_0$, and thus
  \[
  \sE f(Z_n) \leq r + \frac{\sum_{j=0}^{n_0} (f(n)-r) z^j}{\sum_{j=0}^n z^j} \quad \text{for} \ n>n_0.
  \]
  This implies that $\varlimsup\sE f(Z_n)\leq r$, so by letting $r
  \downarrow \varlimsup f(n)$, it follows that $\varlimsup\sE f(Z_n)
  \leq \varlimsup f(n)$. The proof is completed by applying this
  inequality to $-f$.
\end{proof}

\begin{theorem}
  \label{the:necCondLargeMu1}
  Assume $\mu_1\geq\mu_2$, and let $Z_n\sim \geom_n(\mu_1/\mu_2)$. Then
  \begin{align*}
    \varlimsup \sE\lambda(Z_n) &< \mu_2
    \implies S(\lambda,\mu_1,\mu_2) \ \text{is stable},\\
    \varliminf \lambda(n) &> \mu_2
    \implies S(\lambda,\mu_1,\mu_2) \ \text{is unstable}.
  \end{align*}
  Especially, if $\lim \lambda(n)$ exists, then
  \begin{align*}
    \lim\lambda(n) &< \mu_2
    \implies S(\lambda,\mu_1,\mu_2) \ \text{is stable},\\
    \lim \lambda(n) &> \mu_2
    \implies S(\lambda,\mu_1,\mu_2) \ \text{is unstable}.
  \end{align*}
\end{theorem}
\begin{proof}
  The first statement follows from Theorem~\ref{the:suffCond}. To
  prove the second implication, assume $S(\lambda,\mu_1,\mu_2)$ is
  stable. Then by Lemma~\ref{lem:ergodicEquiv}, so is
  $S^N(\lambda,\mu_1,\mu_2)$ for each $N$. Choose an $r\in\R$ so that
  $r<\varliminf\lambda(n)$.  It follows by
  Lemma~\ref{lem:geomDistDiverging} that $\varliminf\lambda(n)\leq
  \varliminf\sE\lambda(Z_n)$. Thus, $\lambda(N)\geq r$ and
  $\sE\lambda(Z_N)\geq r$ for all $N$ large enough. Thus, for all such $N$,
  \begin{align*}
    \sE\lambda(X^N_2)
    &= \sE (\lambda(X^N_2) ; X^N_2 > N) + \sE\lambda(Z_N) \sP(X^N_2\leq N) \\
    &\geq r \sP(X^N_2>N) + r \sP(X^N_2\leq N) = r,
  \end{align*}
  so $\varliminf\sE\lambda(X^N_2)\geq r$. Letting $r$ approach
  $\varliminf\lambda(n)$ we see that $\varliminf \sE \lambda(X^N_2)
  \geq \varliminf \lambda(n)$.  Next, $\lim\sP(X^N_2>0) = 1$, because
  $\sP(X^N_2=0) \leq (\sum_{j=0}^N (\mu_1/\mu_2)^j)^{-1}$
  by~\reff{eq:SStarBound}.  Moreover,
  equality~\reff{eq:SStarStability} shows that
  $\sE\lambda(X^N_2)\leq\mu_2\sP(X^N_2>0)$ for all $N$, so that
  \[
  \varliminf\lambda(n) \leq \varliminf\sE\lambda(X^N_2) \leq
  \varliminf \mu_2\sP(X^N_2>0) = \mu_2,
  \]
  which proves the second claim.  In the special case where
  $\lambda(n)$ has a limit when $n$ tends to infinity,
  Lemma~\ref{lem:geomDistDiverging} shows that
  \[
  \varliminf \lambda(n) = \lim \lambda(n) = \varlimsup \sE \lambda(Z_n),
  \]
  so the last two implications of the theorem now follow from the
  first two.
\end{proof}

There may exist a substantial gap between the necessary and sufficient
stability conditions of Theorem~\ref{the:necCondLargeMu1} if
$\lambda(n)$ is diverging.  To gain some insight why characterizing
the stability of the system is difficult for such $\lambda$, let us
consider the behavior of $S(\lambda,\mu_1,\mu_2)$ as $\mu_1$ tends to
infinity.  Intuition suggests that in this case the system should
resemble the single server queue with service rate $\mu_2$ and
state-dependent input rate $\lambda(n)$, for which it is known
(Asmussen~\cite{asmussen2003}, Corollary 2.5) that stability is
equivalent to
\begin{equation}
  \label{eq:singleServer}
  \sum_{n=0}^\infty \frac{\lambda(0)\cdots \lambda(n)}{\mu_2^{n+1}} < \infty.
\end{equation}
Consider for example the input rates $\lambda(n)=a$ for $n$ even, and
$\lambda(n)=b$ for $n$ odd, where $0<a<b$. Then~\reff{eq:singleServer}
reduces to $\sqrt{ab} < \mu_2$, while with
$Z_n\sim\geom_n(\mu_1/\mu_2)$,
\begin{equation}
  \label{eq:geomBounds}
  \varliminf \lambda(n)
  = a
  < \frac{\mu_1 b + \mu_2 a}{\mu_1+\mu_2}
  = \varlimsup \sE \lambda(Z_n).
\end{equation}
Hence, the gap between the necessary and
sufficient stability conditions in Theorem~\ref{the:necCondLargeMu1}
grows according to
\[
\left[ a,\frac{\mu_1 b + \mu_2 a}{\mu_1+\mu_2} \right]
\longrightarrow
\left[ a, b \right], \quad \mu_1\to\infty.
\]
However, condition~\reff{eq:singleServer} may not in general be the
correct asymptotical stability characterization of
$S(\lambda,\mu_1,\mu_2)$ as $\mu_1\to\infty$, due to a fundamental
difference between the single-server queue and the tandem network.
Namely, if $\lambda(n) = 0$ for some $n$, then the single-server queue
is stable because the queue size cannot exceed $n$.  Obviously, this
property is not true for $S(\lambda,\mu_1,\mu_2)$, and this is why the
necessary and sufficient stability condition for
$S(\lambda,\mu_1,\mu_2)$ must have more complex nature
than~\reff{eq:singleServer}.

\subsection{Eventually vanishing input rate function}
In most applications it is natural to assume that $\lambda(n)$ becomes
eventually zero for large $n$, so that the admission controller
strictly blocks all incoming traffic when the amount of jobs in node 2
becomes too large. In this case $\lim \lambda(n)=0$, so
Theorem~\ref{the:necCondLargeMu1} shows that for $\mu_1\geq \mu_2$,
$S(\lambda,\mu_1,\mu_2)$ is stable regardless of the shape of the
function $\lambda$. On the other hand, if node 1 is the bottleneck,
then Theorem~\ref{the:necCondSmallMu1} determines the stability of the
system, except in the critical case when $\sE\lambda(Z) = \mu_1$. Our
intuition about birth--death processes suggests that the system is
unstable also in this special case.  The validity of this intuition
will be proved next.  The key to proof is the following lemma which
shows that the stability of $S(\lambda,\mu_1,\mu_2)$ implies the
stability of the saturated system $S^*(\lambda,\mu_1,\mu_2)$, where
node~2 behaves as if node~1 never were empty.

\begin{lemma}
  \label{pro:keyEquiv}
  Assume $\mu_1<\mu_2$, and $\lambda(n)=0$ eventually for large $n$.
  If $S(\lambda,\mu_1,\mu_2)$ is stable, then so is the system
  $S^*(\lambda,\mu_1,\mu_2)$ generated by the transition rates
  \[
  q^*(x,y) = q(x,y) \, + \, \mu_1 1(x_1=0, \, y=x+e_2), \quad x\neq y.
  \]
\end{lemma}
\begin{proof}
  Fix a $K\in\Z_+$ so that $\lambda(n)=0$ for all $n>K$, and define
  the transition rates $q'$ by
  \[
  q'(x,y) = q(x,y) \, + \, \mu_1 1(x_1=0, \, x_2>K, \, y=x+e_2), \quad x\neq y.
  \]
  Because $q'(x)<\infty$ for all $x$, the rates $q'(x,y)$ define an
  irreducible Markov process $X'$ on $\Z_+^2\cup\{\kappa\}$.  The
  first step is to show that $X'$ is ergodic.  Note that set of states
  where $q$ and $q'$ differ is now given by $D(q,q') = \{0\}\times
  [K+1,\infty)$.  The key to the proof is to observe that the behavior
  of $X'$ inside $D=D(q,q')$ is similar to a birth--death process with
  birth rate $\mu_1$ and death rate $\mu_2$.  Denote $x=(0,K+1)$.
  Then since $\mu_1<\mu_2$, it follows that for all $y\in
  D\setminus{\{x\}}$,
  \begin{equation}
    \label{eq:keyTime}
    \sE_y T'_x = \frac{y_2-x_2}{\mu_2-\mu_1}.
  \end{equation}
  The ergodicity of $X$ implies $\sE_{x-e_2} T'_D = \sE_{x-e_2} T_D <
  \infty$. Next, since $\sP_{x-e_2}(T'_D \leq T'_x)=1$, we can compute
  using the strong Markov property and~\reff{eq:keyTime},
  \begin{equation}
    \label{eq:timePrime}
    \begin{aligned}
    \sE_{x-e_2} T'_x
    &= \sE_{x-e_2} T'_D + \sE_{x-e_2} (\sE_{X'(T'_D)}T'_x \, ; \ X'(T'_D) \neq x) \\
    &= \sE_{x-e_2} T'_D + \sE_{x-e_2} \, \frac{X'_2(T'_D) - x_2}{\mu_2-\mu_1}\\
    &= \sE_{x-e_2} T_D  + \sE_{x-e_2} \, \frac{X_2(T_D) - x_2}{\mu_2-\mu_1}.
    \end{aligned}
  \end{equation}
  Since $\sE_y T_x = (y_2-x_2)/\mu_2$ for all $y\in
  D\setminus{\{x\}}$, we find in a similar way that
  \begin{equation}
    \label{eq:timeOrig}
    \sE_{x-e_2} T_x = \sE_{x-e_2} T_D + \sE_{x-e_2} \frac{X_2(T_D) - x_2}{\mu_2}.
  \end{equation}
  Since $X$ is ergodic, comparison of~\reff{eq:timePrime}
  and~\reff{eq:timeOrig} shows that $\sE_{x-e_2} T'_x < \infty$.
  Conditioning on the first transition of $X'$ now yields
  \begin{align*}
    \sE_x T'_x
    &= \frac{1}{\mu_1+\mu_2} + \frac{\mu_1}{\mu_1+\mu_2}\sE_{x+e_2}T'_x + \frac{\mu_2}{\mu_1+\mu_2}\sE_{x-e_2} T'_x \\
    &= \frac{1}{\mu_1+\mu_2} + \frac{\mu_1}{\mu_1+\mu_2}\frac{1}{\mu_2-\mu_1}+ \frac{\mu_2}{\mu_1+\mu_2}\sE_{x-e_2} T'_x,
  \end{align*}
  showing that $\sE_x T'_x < \infty$. By irreducibility, it now follows
  that $X'$ is ergodic.

  Finally, note that the set $\overline{D}(q',q^*) \subset [0,1]
  \times [0,K+1]$ is finite. Thus, in light of
  Lemma~\ref{lem:ergodicEquiv} we may now conclude that the Markov
  process $X^*$ generated by $q^*(x,y)$ is ergodic.
\end{proof}

\begin{theorem}
  \label{the:necSuffCond}
  Assume that $\lambda(n) = 0$ eventually for large $n$.
  \begin{enumerate}[(i)]
  \item If $\mu_1<\mu_2$, $S(\lambda,\mu_1,\mu_2)$ is stable if and only if
    $\, \sE\lambda(Z)<\mu_1$ with $Z\sim\geom(\mu_1/\mu_2)$.
  \item If $\mu_1\geq\mu_2$, $S(\lambda,\mu_1,\mu_2)$ is always stable.
  \end{enumerate}
\end{theorem}
\begin{proof}
  In light of Theorems~\ref{the:necCondSmallMu1}
  and~\ref{the:necCondLargeMu1}, all we need to show is that the
  stability of $S(\lambda,\mu_1,\mu_2)$ implies $\sE \lambda(Z) <
  \mu_1$ when $\mu_1<\mu_2$. So, assume $\mu_1<\mu_2$ and that
  $S(\lambda,\mu_1,\mu_2)$ is stable.  By Lemma~\ref{pro:keyEquiv}, so
  is $S^*(\lambda,\mu_1,\mu_2)$. From the balance equations for $X^*$
  it is easy to see that $X^*_2\sim\geom(\mu_1/\mu_2)$. Thus the
  stationary mean rate of jobs arriving to node~1 equals $\sE
  \lambda(Z)$, while the corresponding rate out is equal to $\mu_1
  \sP(X^*_1>0)$. Because these two quantities must be equal in a
  stable system, we conclude that
  \[
  \sE \lambda(Z) = \mu_1 \sP(X^*_1>0) < \mu_1,
  \]
  where the last inequality is strict because $\sP(X^*_1=0) > 0$ by
  the ergodicity of $X^*$.
\end{proof}

\section{Sensitivity analysis of the stability region}
This section focuses on the stability of the system subjected to
fluctuations in the system parameters.  The treatment here is
restricted to the case of eventually vanishing input rates, where
Theorem~\ref{the:necSuffCond} completely characterizes the stable
parameter region.

\subsection{Sensitivity with respect to varying service rates}
The next proposition shows that with nonincreasing input rates, the
stability of the system is preserved under speeding up of node 1.
\begin{proposition}
  \label{pro:speedServer1}
  Assume $\lambda$ is nonincreasing and $\lambda(n)=0$ eventually for
  large $n$.  Then for all $\mu_1'\geq\mu_1$,
  \[
  S(\lambda,\mu_1,\mu_2)  \ \text{is stable} \implies
  S(\lambda,\mu_1',\mu_2) \ \text{is stable}.
  \]
\end{proposition}
\begin{proof}
  Assume that $S(\lambda,\mu_1,\mu_2)$ is stable and let $\mu_1'\geq
  \mu_1$.  If $\mu_1'\geq \mu_2$, then $S(\lambda,\mu_1',\mu_2)$ is
  stable by Theorem~\ref{the:necSuffCond}. On the other hand, if
  $\mu_1'<\mu_2$, then also $\mu_1<\mu_2$, and the necessary condition
  of Theorem~\ref{the:necSuffCond} shows that $f(\mu_1/\mu_2)<\mu_1$,
  where
  \[
  f(x) = (1-x)\sum_{n=0}^\infty \lambda(n) x^n.
  \]
  Because the sequence $\lambda(n)$ is bounded and nonnegative, $f$
  is differentiable in $(0,1)$ with
  \[
  f'(x) = \sum_{n=0}^\infty (n+1) \left( \lambda(n+1)-\lambda(n)
  \right) x^n \leq 0,
  \]
  so that $f(\mu_1'/\mu_2) \leq f(\mu_1/\mu_2)$. It follows that
  $f(\mu_1'/\mu_2) < \mu_1'$, which guarantees the stability of
  $S(\lambda,\mu_1',\mu_2)$ by Theorem~\ref{the:necSuffCond}.
\end{proof}

To see why it is necessary to require $\lambda$ to be nonincreasing,
consider the following example.
\begin{example}
  \label{exa:nonMonotonic}
  Let $\mu_2=1$, and assume that $\lambda(n)=0$ for $n\geq 3$. Then
  $S(\lambda,\mu_1,\mu_2)$ is stable for all $\mu_1\geq 1$, and for $\mu_1\in (0,1)$, the
  stability of $S(\lambda,\mu_1,\mu_2)$ is equivalent to
  \begin{equation}
    \label{eq:nonMon}
    \mu_1^{-1} (1-\mu_1) \left(\lambda(0)+\lambda(1) \, \mu_1 +\lambda(2) \, \mu_1^2 \right) < 1.
  \end{equation}
  Figure~\ref{fig:exampleNonMonotonic} shows the the left-hand side
  of~\reff{eq:nonMon} as a function of $\mu_1$, where
  $\lambda(0)=\lambda(1)=\frac{1}{100}$ and $\lambda(2) = 5$.  The
  plot illustrates that by increasing the service rate $\mu_1$ from
  $\frac{1}{5}$ to $\frac{1}{2}$ destabilizes the system.
  \begin{figure}[h]
    \begin{center}
      \includegraphics[height=.2\textheight]{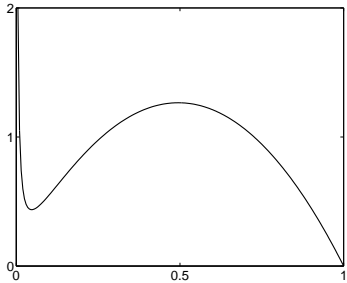}
      \caption{The left-hand side of~\reff{eq:nonMon} as a function of $\mu_1$.}
      \label{fig:exampleNonMonotonic}
    \end{center}
  \end{figure}
\end{example}

Alternatively, we may fix $\mu_1$ and see what happens when $\mu_2$
varies. The following proposition tells a rather surprising result:
Even with nonincreasing $\lambda$, acceleration of one of the servers
may indeed destabilize the system.  The physical intuition behind
Proposition~\ref{pro:speedServer2} is that when $\mu_2$ is very large,
the admission controller finds node 2 empty most of the time.  This
means that the input rate to the system is close to $\lambda(0)$.

\begin{proposition}
  \label{pro:speedServer2}
  Assume $\lambda$ is nonincreasing and $\lambda(n)=0$ eventually for
  large $n$ and fix $\mu_1>0$. Then
  \begin{itemize}
  \item for $\lambda(0)\leq\mu_1$, $S(\lambda,\mu_1,\mu_2)$ is stable for all $\mu_2>0$,
  \item for $\lambda(0)>\mu_1$, $S(\lambda,\mu_1,\mu_2)$ becomes
    eventually unstable for large $\mu_2$.
  \end{itemize}
\end{proposition}
\begin{proof}
  Observe first that by Theorem~\ref{the:necSuffCond},
  $S(\lambda,\mu_1,\mu_2)$ is stable for all small $\mu_2\leq \mu_1$.
  To study the case with $\mu_2>\mu_1$, fix a number $n_0$ so that
  $\lambda(n)=0$ for all $n>n_0$. Then with $Z\sim \geom(\mu_1/\mu_2)$,
  \begin{equation}
    \label{eq:geomMean}
    \sE \lambda(Z) = (1-\mu_1/\mu_2) \sum_{n=0}^{n_0} \lambda(n) (\mu_1/\mu_2)^n.
  \end{equation}
  If $\lambda(0)\leq \mu_1$, then~\reff{eq:geomMean} implies that for
  all $\mu_2>\mu_1$,
  \[
  \sE \lambda(Z) \leq \lambda(0) (1-(\mu_1/\mu_2)^{n_0+1}) < \mu_1,
  \]
  which by Theorem~\ref{the:necSuffCond} is sufficient for stability.
  Moreover, the right-hand side of~\reff{eq:geomMean} converges to
  $\lambda(0)$ as $\mu_2\to\infty$. From this we can conclude that if
  $\lambda(0)>\mu_1$, then $\sE \lambda(Z) > \mu_1$ for large enough
  values of $\mu_2$. By Theorem~\ref{the:necSuffCond},
  $S(\lambda,\mu_1,\mu_2)$ is unstable for such $\mu_2$.
\end{proof}

\subsection{Phase partition for threshold-based admission control}
Consider the network with threshold-based admission control, and
assume without loss of generality that jobs arrive to the network at
unit rate. Denoting the threshold level by $K$, this system can be
modeled as $S(\lambda,\mu_1,\mu_2)$ with $\lambda(n) = 1(n\leq K)$.
Theorem~\ref{the:necSuffCond} now implies that for each
$K\in\Z_+\cup \{\infty\}$, the set of $(\mu_1,\mu_2)$ for which the system
is stable equals
\[
R_K = \{(\mu_1,\mu_2): \ 1-(\mu_1/\mu_2)^{K+1} < \min(\mu_1,\mu_2) \}.
\]
Since $R_K\supset R_{K+1}$ for all $K$, the stabilizable region is
given by $\cup_{K\leq\infty} R_K = R_0$, while $R_\infty =
\{(\mu_1,\mu_2): \min(\mu_1,\mu_2)>1\}$ represents the system with no
overload.  The positive orthant of $\R^2$ can now be partitioned into
four phases as follows:
\begin{itemize}
\item $A_1 = R_\infty$ is the region where the uncontrolled
  system is stable.
\item $A_2 = \cap_{K<\infty} R_K$ represents the region where any control
  stabilizes the overloaded system.
\item $A_3 = R_0 \setminus \cap_{K<\infty} R_K$ is the region where
  the overloaded system is stabilizable using strict enough admission
  control.
\item $A_4 = R_0^c$ is the region where the system cannot be stabilized.
\end{itemize}
\begin{figure}[h]
  \begin{center}
    \psfrag{mu1}{$\mu_1$}
    \psfrag{mu2}{$\mu_2$}
    \psfrag{A1}{$A_1$}
    \psfrag{A2}{$A_2$}
    \psfrag{A3}{$A_3$}
    \psfrag{A4}{$A_4$}
    \includegraphics[width=.5\textwidth]{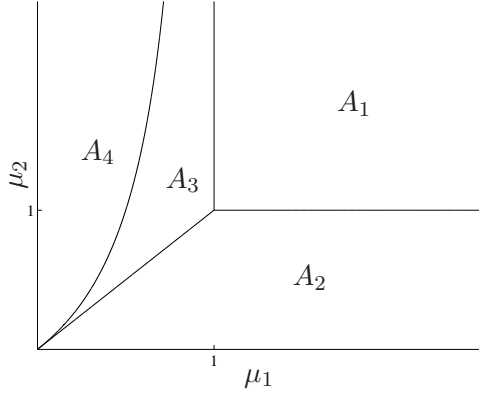}
    \caption{Phase diagram for threshold-based admission control.}
    \label{fig:phasePartition}
  \end{center}
\end{figure}
This partition is depicted in Figure~\ref{fig:phasePartition}. The
phase diagram clearly illustrates the content of
Propositions~\ref{pro:speedServer1} and~\ref{pro:speedServer2},
showing that accelerating server 1 drives the system towards more
stable regions, while speeding up server 2 may destabilize the
network.

\section{Conclusion}
This paper considered the problem of characterizing the stability
region of a two-node queueing network with feedback admission control.
For eventually vanishing input rates, the characterization was shown
to be complete.  It was also illustrated how the presence of feedback
signaling breaks down some typical monotonicity properties of queueing
networks, by showing that increasing service rates may destabilize the
network.

For a diverging input rate function and bottleneck at node 2, the
exact characterization of the stability region remains an open
problem. Other possible directions for future research include
generalizing the results for nonexponential service and inter-arrival
times, and considering queueing networks with more than two nodes.

\bibliographystyle{apalike}
\bibliography{lslReferences}

\section*{Acknowledgements}
This work was funded by the Academy of Finland Teletronics II /
FIT project and the Finnish Graduate School in Stochastics. The author
would like to thank Ilkka Norros for his valuable help and advice
during the project, Jarmo Malinen for many inspiring discussions, and
the anonymous referee for helpful comments on improving the
presentation of the results.

\end{document}